\documentclass{amsart}

\usepackage{amsmath, amsthm}
\usepackage{amssymb, latexsym}
\usepackage{amsfonts}

\newtheorem{theorem}[equation]{Theorem}

\newtheorem{proposition}[equation]{Proposition}

\newtheorem{lemma}[equation]{Lemma}

\newtheorem{Example}[equation]{Examples}

\theoremstyle{definition}

\newcommand{\thmref}[1]{\rm Theorem~\ref{#1}}

\newcommand{\propref}[1]{\rm Proposition~\ref{#1}}

\newcommand{\beql}[1]{\begin{equation}\label{#1}}
\newcommand{\eeq} {\end{equation}}

    \font\Aaa=msam10

\font\Bbb=msbm10

\newcommand\R{\hbox{\Bbb R}}
\newcommand\Z{\hbox{\Bbb Z}}

\newcommand\F{\hbox{\Bbb F}}

\font\Aaa=msam10

\def\qed{\hbox{~~\Aaa\char'003}}

\font\Bbb=msbm10

\def\R{\hbox{\Bbb R}}

\def\Z{\hbox{\Bbb Z}}

\def\F{\hbox{\Bbb F}}

 \def\CC{\hbox{\Bbb C}}

\numberwithin{equation}{section}

\newcommand\G{\Gamma }
\newcommand\D{{ \Delta }}
\newcommand\z{{ \zeta }}

        \def\<{{\langle}}
        \def\>{{\rangle}}

        \font\Aaa=msam10

\font\Aaa=msam10

\font\Bbb=msbm10

\def\R{\hbox{\Bbb R}}

\def\Z{\hbox{\Bbb Z}}

\def\F{\hbox{\Bbb F}}

\def\a{\alpha}

\def\div{ \kern-.5pt\hbox{\big |} }
\def\ndiv{ {\not\kern-.5pt\hbox{\big |}\,} }
\def\ndivv{ {\not\kern+1.5pt\hbox{$\mid$}\,} }

\def\B{^2\kern-.8pt B}
\def\G{^2\kern-.8pt G}
\def\EH{^2\kern-.8pt\hat  E}
\def\E{^2\kern-.8pt E}
\def\D{^3\kern-1pt D}
\def\FF{^2\kern-.8pt F}

\newdimen\refcodesize
\newbox\seriesbox
\def\proof{\noindent {\bf Proof.~}}

\DeclareRobustCommand{\SkipTocEntry}[4]{}
\begin{document}

\title[MUBs from bent functions]
{MUBs from bent functions
 }
 
    \author{William M. Kantor }
   \address{Brook House, Brookline, MA 02445}
    \email{kantor@uoregon.edu}

\begin{abstract}

This note  contains  a  simple construction of complete sets of MUBs, using bent functions
to write the new  basis vectors as explicit linear combinations of the standard basis.   

% \vspace{-18pt}  
 \end{abstract}
\maketitle 

%\vspace{-2pt}

\section{Introduction}
 \label{Introduction} 
This note  contains  a short,  elementary construction of  complete sets of mutually unbiased bases (MUBs)
 in $\CC^N,$ 
where $ N=p^n$ for
a prime   $p$.  This   sequel to \cite{Ka} uses an idea presented there 
without  the connection to bent functions  being noticed.   
 Preferring brevity~we refer to 
 \cite{Wo,Ka} for background and some references concerning complete sets of  MUBs.
 
 Our point of view   is different from the  groups/geometries approaches to MUBs
 \cite{WoF,Wo,CCKS,Ka}.   However, this has not led to the construction of any new complete sets of MUBs.
  
%\vspace{-2pt}

\section{Odd characteristic} 
\rm

\label{MUBs from bent functions}
Using row vectors, let $ V=\Z_p^n$ with $p$ an odd  prime,   equipped with the usual~dot product;
and equip $\CC^N,$ $N=  p^n$, with the usual hermitian 
inner product
and orthonormal basis   $\{e_a\mid a\in V\}$.
Orthonormal bases 
$\mathfrak{M}$ and $\mathfrak{M}'$  
are \emph{mutually unbiased bases} (MUBs) if 
$|(u,v)|=1/\sqrt{N}$  for all  $u\in \mathfrak{M},v\in   \mathfrak{M}'$.
A {\em complete set of} MUBs in $\CC^N$ is a set of $N+1$ pairwise mutually unbiased orthonormal bases.

Let $\beth$ be a  {\em mubent set of functions}  $V\to \Z_p$:
  a set of $|V|$  functions $V\to  \Z_p$ such that the difference of any two is bent.   
  A defining property of a {\em bent function}
 ${B\!:\!V\to  \Z_p}$  \cite[p. 75]{Me}  
 is   that, if $0\ne u\in V$, then $v\mapsto B(v+u)-B(v)$ takes every value in $\Z_p$
 exactly $|V|/p$ times.

If  $a\in V, $  $B\in \beth$, let 
$e_{a,B}:=  
\displaystyle
\sum_{v\in V}\z^{a\cdot v+B(v)}e_{v}$, where $\z\in \CC$ is a primitive $p$th root of 1.
This   definition  comes from  coding theory, finite symplectic geometry 
and finite affine planes  \cite[(5.3)]{CCKS}
 (compare \cite[(14)]{WoF}).

Let   ${\mathfrak{M}}_\infty:=\{e_a \mid a\in V\}$ and 
${\mathfrak{M}}_B:=\{\frac{1}{\sqrt N} e_{a,B}\mid  a \in V\}$ for $B\in \beth$.

\begin{theorem}
\label{main bent} 
If $\,\beth$ is a    mubent set of functions  $V\to \Z_p$   then
$\{{\mathfrak{M}}_\infty \}\cup\{{\mathfrak{M}}_B\mid B\in \beth\}$ is a complete set of MUBs in $ \CC^N.$
 \end{theorem}
  
  \proof 
  For $a,a'\in V$ and $B,B'\in \beth$ let
$$\alpha:= (e_{a,B}, e_{a',B'})=\displaystyle
\sum_{v\in V}
\z^{a \cdot v+  B(v)-a' \cdot v - B'(v)} =
\sum_{v\in V}
\z^{d \cdot v +  \Delta(v)}
$$
with  {$d:=a-a'$  and $ \Delta:=\! B-B'$}.  Here
  ${(e_{a,B}, e_{a',B})=(|V|/p)\sum_{j=0}^{p-1}\z^j=0}$~when ${a\ne a'}$,
and  $(e_{a,B}, e_{a,B})\! =\! |V| =\!  N$.
  
%\medskip\medskip
    Let  $\Delta\ne 0$, so $\Delta$ is bent by hypothesis.  Use $u=v'-v$ in  the following   calculation: 
\vspace{2pt}

$\begin{array}{llllllll}
\vspace{2pt}  
|(e_{a,B}, e_{a',B'})|^2
=\,  \alpha\bar\alpha \,  
\hspace{-6pt}&= & \hspace{-10pt}\vspace{2pt}\displaystyle
 \sum_{v,v'\in V}
\z^{d \cdot  v'  +\Delta(v')-d \cdot  v - \Delta(v) } 
\\  
\vspace{-2pt}  \hspace{-6pt}&= & \hspace{-6pt}\displaystyle
 \sum_{u \in V}\z^{d\cdot u  }
 \sum_{v \in V}
\z^{ \Delta(v+u)-\Delta(v)}  \vspace{-2pt}
\\  
\vspace{2pt}\hspace{-6pt}&= & \hspace{-6pt}\displaystyle  
\sum_{ u =0}\z^{d\cdot u} { \sum_{v \in V}\z^0}+ 
 \displaystyle
 \! \!\sum_{0\ne u \in V}\z^{d\cdot u  }
  (|V|/p)   \sum_{j=0}^{p-1}\z^j
\\  
\hspace{-6pt}&= & \hspace{-6pt}   |V| +0\,= \,N     .
\end{array}
$
\smallskip

 \noindent 
Then
 $ |(\frac{1}{\sqrt N}e_{a,B}, \frac{1}{\sqrt N}e_{a',B'})|^2 =\frac{1}{ {N^2}}N=\frac{1}{ {N}},$
 as required for MUBs. \qed

%\medskip
 \Example \rm    \label{odd}
The  simplest  examples  use $V=\F_{p^n},$   the trace map 
$Tr\!:\!V\to  \Z_p$, and $\beth :=\{ x\mapsto Tr(ax^2)\mid a\in V\}$.

 \smallskip\smallskip
The preceding calculation  slightly generalizes one in \cite[p.~032204-5]{Ka}  that used
    quadratic  functions   $v  Mv^T/2$ for symmetric $n\times n$ matrices $M$.
 If $\beth$ consists of such quadratic   functions  then the mubent set condition states that the set of 
 these matrices $M$ is a spread set.
 (A {\em spread set}  of $n\times n$ matrices over $\Z_p$ 
  is a set of $p^n$ such matrices     for which   all
  differences between pairs of these matrices  are nonsingular,  and such sets  produce affine planes
  \cite[p.~220]{De}.)
     Many mubent sets 
     are closed under addition, but some are not \cite[Example 3.5(e)]{Ka}.%
     
    The preceding construction fails when $p=2$ because there is then no suitable set $\beth$  (the next section avoids this problem by modifying the meaning of ``mubent'').
   A set of bent functions $\Z_2^n\to \Z_2$   such that the difference of any two is bent has size at most $2^{n-1}$.  (Such a set produces a set of   MUBs in $\R^{2^n}$ 
   (using $\z=-1$ in~the above construction),  which 
      has size at most $2^{n-1}+1$  {\cite[(3.9)]{CCKS}.)%  

%\vspace{-2pt}
\section{Characteristic $2$} 
\rm
    
\label{main char 2} 

For 2-power dimensions our construction
of complete sets of MUBs is similar to the preceding one, using both $\Z_2$ and $\Z_4$   instead of just $\Z_p$
(as in \cite{WoF}).

We use $ V=\Z_2^n$,  $N=  2^n$, $\CC^N$  and   $\{e_a\mid a\in V\}$ as before .

Let $\beth$ be a {\em mubent set of functions}  $V\to \Z_4$:
a set of $|V|$   functions $V\to  \Z_4$ such that the difference of any two is a bent function
 $V\to  \Z_4$.   
This time our  definition of a bent function
 $B\!:\!V\to  \Z_4$  (compare~\cite[p.~403]{Me})  is that, if $0\ne u\in V$,  then
 $n(u,0) \!= \!n(u,2)$ and $n(u,1) \!= \!n(u,3)$, 
where $n(u,k) \!:= \!|\{ v\in V \!\mid \! B(v+u)-B(v) \!= \!k\}|$  for   $k\in \Z_4$.

If $a\in V, $  $B\in \beth$, let 
$e_{a,B}:=   
\displaystyle
\sum_{v\in V}(-1)^{a\cdot v}i^{B(v)}e_v $  (compare~\cite[(36)]{WoF}).

Let   ${\mathfrak{M}}_\infty:=\{e_a \mid a\in V\}$ and 
${\mathfrak{M}}_B:=\{\frac{1}{\sqrt N} e_{a,B}\mid  a \in V\}$ for $B\in \beth$.

\begin{theorem}
\label{char 2  bent} 
If $\beth$ is a  mubent set  of functions $V\to \Z_4$ then
$\{{\mathfrak{M}}_\infty \}\cup\{{\mathfrak{M}}_B\mid {B\in \beth}\}$ is a complete set of MUBs in $ \CC^N.$
 \end{theorem}

  \proof 
  For $a,a'\in V$ and $B,B'\in \beth$   let   
$$\alpha:= (e_{a,B}, e_{a',B'})=\displaystyle
\sum_{v\in V}
(-1)^{ a \cdot  v-a' \cdot  v}i^{  B(v)  - B'(v)} =
\sum_{v\in V}
(-1)^{ d \cdot   v }i^{    \Delta(v)}
$$
with  {$d:= a-a'$  and $ \Delta:=B-B'$}.  Here
   $(e_{aB}, e_{a'B})=0$ when $a\ne a'$ and
 $(e_{aB}, e_{aB})= N$.
 
%\medskip\medskip
    Let  $\Delta\ne 0$.  Use  $u=v'- v$ in  the following  calculation: 

\vspace{2pt}

$\begin{array}{llllllll}
\vspace{2pt}  
|(e_{a,B}, e_{a',B'})|^2
\,  =\,  \alpha\bar\alpha 
\hspace{-6pt}&= & \hspace{-10pt}\vspace{2pt}\displaystyle
 \sum_{v,v'\in V}
(-1)^{d \cdot v' - d \cdot v}i^{ \Delta(v') - \Delta(v) } 
\\  
\vspace{-2pt} \vspace{2pt}  \hspace{-6pt}&= & \hspace{-6pt}\displaystyle
 \sum_{u \in V}(-1)^{d\cdot u  }
 \sum_{v \in V}
i^{ \Delta(v+u)-\Delta(v)}  
\\  
\vspace{1pt}  \hspace{-6pt}&= & \hspace{-6pt}
\displaystyle \sum_{ u =0}(-1)^{d\cdot u}{ \sum_{v \in V}i^0}
+ \! \!\sum_{0\ne u \in V}(-1)^{d\cdot u  }\sum_{ k =0}^3 n(u,k)i^k   
\\   
\hspace{-6pt}&= & \hspace{-6pt}   |V| +0\,= \,N     
\end{array}
$
\vspace{2pt}

 \noindent since  $\Delta$ is bent    $V\to  \Z_4$,  so
 $n(u,0)i^0+ n(u,2)i^2=0$ and  $n(u,1)i^1+ n(u,3)i^3=0$ for $u\ne0$. 
 Then
 $ |(\frac{1}{\sqrt N}e_{a,B}, \frac{1}{\sqrt N}e_{a',B'})|^2 =\frac{1}{ { N^2}}N= \frac{1}{ N},$
 as required for MUBs. \qed
\rm
%\medskip\medskip 

\medskip

The remainder of this section presents known constructions  of complete sets of MUBs using \thmref{char 2  bent} together with quadratic bent functions $V\to \Z_4$  and spread sets \cite{CCKS,BBRV}.

 If $v\in V=\Z_2^n$ let $\hat v\in \Z_4^n$ have the same coordinates as  $v$ but with 0 and 1 viewed in $\Z_4$.   
 Similarly, any matrix $R$ over $\Z_2$ produces a matrix $\hat R$ over $\Z_4$.%
   
\lemma
 Any   symmetric $n\times n$ matrix $M$  over $\Z_4$ such  that {\rm$M$~(mod 2)} is nonsingular
 produces a quadratic bent function $B_M\!:\!V\to \Z_4$   via $B_M(v):=\hat v   M \hat v^T.$
 
 \rm
 \medskip
 
 \proof
 Let $B:=B_M$.  Claim: $B(  u+  v)\!=\!B(  u)+B(  v)+2 \hat u  M  \hat v^T$ for all 
 ${u,v\in \Z_2^n}$ (compare~\cite[(4.2)]{CCKS}, \cite[p.~403]{Me}).  
 \vspace{-2pt}
 If $a,b\in \Z_2$  then, trivially,
 $\widehat{a+b}=\hat a+\hat b+2\hat a\hat b$.   
 If $u=(a_i), v=(b_i)\in V$ let   $\hat u*\hat v  :=(\hat a_i\hat b _i)\in \Z_4^n$.   Then
 $\a:=\widehat {u+ v}M\widehat{u + v}^T$  is  
 $({\hat u+\hat v+{2\hat u*\hat v}})M(\hat u+\hat v+2\hat u*\hat v)^T =(\hat u+\hat v)M(\hat u+\hat v)^T+(\hat u+\hat v)M(2\hat u*\hat v)^T + {(2\hat u*\hat v)M(\hat u+\hat v)^T }.$
  Since  $M$ is symmetric, 
 $\hat u M \hat v^T = \hat v M \hat u^T $ and 
 $(\hat u+\hat v)M(2\hat u*\hat v)^T =(2\hat u*\hat v)M(\hat u+\hat v)^T.$  Then
 $\a=(\hat uM\hat  u^T+\hat vM\hat v^T+2\hat uM\hat v^T)+0$, proving our claim.

 Given $0\ne u\in V$ and  $k\in \Z_4$ we need the number of $  v\in V$ such  that $B(v+u)-B(v)=k$; 
 that is,  such  that
$2\hat u    M  \hat v^T=k-B(  u)  .$  There are 
 $n(u,k)=n(u,k+2)=0$  solutions if  $k-B(  u) $ is 1 or  3, so assume that
 $k-B(  u) =2s$ for $s=0$ or 1.   
 Let $M':=M$ (mod 2).
  If $v\in V$ then $2\hat u    M  \hat v^T=2s$ if and only if $ u    M' v^T=s$.
 Since    $M'$  is nonsingular, $uM'\ne0$ and 
  $v\mapsto (u   M' ) v^T$   takes every value in $\Z_2$ exactly $|V|/2$ times.
 Then   $n(u,k)=n(u,k+2)=|V|/2$, and $B\!:\!V\to \Z_4$ is bent.\qed

\proposition \label{bentness}
Every spread set $\Sigma$ of symmetric $n\times n$ matrices over $\Z_2$ produces a mubent set 
\vspace{1pt}
$\{B_{\hat R  },0 \mid R\in \Sigma\backslash\{0\}\}$
of  quadratic functions $\Z_2^n\to \Z_4$ and   a complete set of MUBs in $\CC^{2^n}\!$.
 
 \rm
 \medskip

\proof
Let $B_\bullet\in \Sigma$ and replace $\Sigma $ by the spread set $ \{B-B_\bullet\mid B\in \Sigma\}.$
Now $0\in \Sigma$,  and  $\Sigma\backslash\{0\}$  consists of 
nonsingular symmetric matrices.
If $R $ and $S $ are distinct  members of   $ \Sigma\backslash\{0\}$  then   
$\hat R-\hat S$ is a symmetric matrix that is  nonsingular (mod 2) since $\Sigma$ is a spread set.   
By the preceding lemma, $B_{\hat R -\hat S}$ is a bent function $\Z_2^n\to \Z_4$.
Since $B_{\hat R  }-B_{ \hat S}=B_{\hat R - \hat S} $ by definition, 
\vspace{2pt}
$\{B_{\hat R  },0\mid R\in \Sigma\backslash\{0\}\}$    is a mubent set 
of  quadratic functions, and \thmref{char 2  bent} applies.\qed

  \Example \rm 
  \label{even}
View $V$ as $\F_{2^n}$, with trace map  $Tr\!:\! V\to \Z_2$  and  nondegenerate symmetric
 bilinear form $Tr(xy)$ for $x,y\in V$.   Choose an orthonormal basis for $V$,  so  $Tr(xy)=x\cdot y$   for $x,y\in V$. 
 Then a self-adjoint  $\Z_2$-linear transformation $R$ on $V$
 can be viewed as a symmetric matrix, producing a matrix  $\hat R$ over $\Z_4$.

The  set  $\{cI\mid c\in V\}$ of self-adjoint transformations of  $V$ is a spread set.
\propref{bentness} produces an associated 
  mubent set    $\beth :=\{B_{\widehat { cI \hspace{.5pt}}},0 \mid   c \in V\backslash\{0\}\}$ and a complete set  of MUBs:
  $\{{\mathfrak{M}}_\infty \}\cup\{{\mathfrak{M}}_B\mid {B\in \beth}\}$.
    
\section {Remarks}\rm
Examples~\ref{odd} and \ref{even} are    the most familiar complete sets of MUBs \cite{WoF}.   Others are in \cite[Ex. 2.5, 2.7, 3.5, 3.7]{Ka}.

For $p>2$ all known   mubent sets of functions  $V\to \Z_p$    consist of quadratic functions, except for
sets  obtained using \cite{CM}  (see \cite[Ex.~3.7]{Ka}).

All  known    mubent sets of functions  $V\to \Z_4$    consist of   quadratic   functions.

It is not clear how to use 
Theorems~\ref{main bent} or \ref{char 2  bent} 
to deal with   the unitary equivalence of  complete sets of MUBs  in $\CC^N$.   
  When $\beth$ consists of quadratic functions the 
  groups and geometries
  approach in \cite{CCKS,Ka} uses Heisenberg groups  to
turn unitary equivalence   into a question  about   isomorphisms of associated finite affine planes.
There are many known pairwise inequivalent  complete sets of MUBs for any
$p $  and suitable   $N$.
When $p>2$ the number  of known complete sets  is $O(N)$, whereas when $p=2$  that number   
 is not bounded above by any polynomial in $N$.%

   \medskip\medskip 
   
   \noindent{\em Acknowledgement.} I am grateful to Bill Martin for invaluable  help   
   in the preparation of
     this note.

\end{document}